\documentclass[12pt]{article} 
\usepackage[letterpaper,left=2.5cm,top=2.5cm,right=2.5cm,bottom=2.5cm]{geometry}  
\usepackage[none]{hyphenat}
\usepackage{latexsym} 
\usepackage{layout}
\usepackage{float}
\usepackage{textcomp}
\usepackage{amstext}
\usepackage{multicol} 
\usepackage{longtable} 
\usepackage{lscape} 
\usepackage{multirow} 
\usepackage{array} 
\usepackage{amssymb}
\usepackage{amsmath} 
\usepackage{amsthm}
\usepackage[T1]{fontenc} 
\usepackage{graphicx}
\usepackage{verbatim}

\newtheorem{theorem}{Theorem}[section]
\newtheorem{lemma}[theorem]{Lemma}

\theoremstyle{definition}
\newtheorem{definition}[theorem]{Definition}

\theoremstyle{remark}

\theoremstyle{plain}
\newtheorem{prop}[theorem]{Proposition}

\usepackage[pdftex]{color}
\usepackage[pdftex,colorlinks=true,linkcolor=blue,citecolor=black,urlcolor=cyan]{hyperref}

\newlength{\LL}

\title{ \Huge \bf An Integral Approach to Prescribing Scalar Curvature Equations \footnote{This work was supported by NSFC (No. 12141103)}}
\author{Chen Ruosi \footnote{crs22@mails.tsinghua.edu.cn}
~,~Jian Huaiyu \footnote{hjian@mail.tsinghua.edu.cn}
~,~Zhou Xingchen  \footnote{zxc3zxc4zxc5@stu.xjtu.edu.cn}
~~ \\ {\it Department of Mathematical Sciences} \\ {\it Tsinghua University}
}

\date{\today}

\begin{document}

\maketitle

\hrulefill

\begin{abstract}
We develop an integral approach to obtain interior a priori $C^{1,1}$ estimates for convex solutions of prescribing scalar curvature equations $\sigma_2(\kappa) = f(x)$ as well as the Hessian equations $\sigma_2(D^2u) = f(x)$. This new approach can deal with the case when $f$ is of weaker regularity. As a result, we prove that the $C^{1,1}$ modules of the  solutions depend only on  the Lipschitz modules of  $f(x)$, instead of the  $\|f\|_{C^k}$ for some $k\geq 2$ in all the papers we have known up to now.
\end{abstract}

\section{Introduction}\label{sec1}
Consider a $C^2$ isometrically immersed hypersurface $M$ in $\mathbb {R}^{n+1}$ with position vector $X=(x,u(x))$. Denote its  principle curvatures  $\kappa=(\kappa_1,\kappa_2,\cdots, \kappa_n)$. The problem of  prescribing curvature surfaces relates to a long standing research on the following general equations,
\begin{eqnarray*} \sigma_m(\kappa) = f(x, u ,\nabla u), \end{eqnarray*}
where $\kappa=(\kappa_1,\kappa_2,\cdots, \kappa_n)$ and $\sigma_m(\kappa)$ is the elementary symmetric polynomials with respect to $\kappa_1,\kappa_2,\cdots, \kappa_n$ of $m$-order. One can see, for example, the Minkowski problem considered by Cheng-Yau \cite{MR4519019}, Nirenberg \cite{MR0058265} and Pogorelov \cite{MR0478079}, the prescribing general Weingarten curvature problems by Alexandrov \cite{MR0086338} and  Guan-Guan \cite{MR1933079}, and the prescribing curvature measure problems in convex geometry by Guan-Li-Li \cite{MR2954620} and  Guan-Lin-Ma \cite{MR2507106}.

In this paper, we are concerned with the case $m=2$, that is the prescribing scalar curvature problem, and the corresponding equation is read as
\begin{eqnarray}
    \label{eq:sigma_kappa}
    \sigma_2(\kappa) = f(x).
\end{eqnarray}
We will also study its analogue, the well-known $2$-Hessian equation
\begin{eqnarray}\label{eq:sigma_lambda}
\sigma_2(D^2u):=\sigma_2(\lambda_u) = f(x),
\end{eqnarray} where $\lambda_u=(\lambda_1, \cdots, \lambda_n)$ are the eigenvalues of the Hessian matrix $D^2u$.
The interior a priori $C^{2}$ estimates means to answer the
the following question  positively:
\begin{itemize}
    \item [\textbf{I}] For smooth solutions to equations (\ref{eq:sigma_kappa}) or (\ref{eq:sigma_lambda}), for positive mean curvature or $\Delta u$, can we have the interior a priori Hessian estimates
    \begin{eqnarray*}
    |D^2u(0)|\le C(n, \|u\|_{C^1(B_1(0))},\|f^{-1}\|_{L^\infty(B_1(0))},\|f\|_{C^k(B_1(0))})
    \end{eqnarray*} for some number $k\geq 0$?
\end{itemize}  Here and below, $C(\cdots)$ denotes a positive constant depending only on the quantities $\cdots$.

This question has been studied extensively and there is a lot of positive answers, which we are going to
list.   For equation (\ref{eq:sigma_kappa}) and (\ref{eq:sigma_lambda}) in dimension two it was obtained by Heinz \cite{heinz1959elliptic} and Pogorelov \cite{pogorelov1964}. For general dimension,  Pogorelov-type estimates were obtained in Chow-Wang \cite{MR1835381} for equation (\ref{eq:sigma_lambda}) with 2-convex boundary condition, and Sheng-Urbas-Wang \cite{ShengUrbasWang04} with affine boundary condition for equation (\ref{eq:sigma_kappa}), but their estimates depend on $\|f\|_{C^2}$. It was showed in Guan-Qiu \cite{guan2019interior}  that the Hessian bound can be obtained under certain convex conditions $\sigma_3\ge-C$, by a pointwise argument using maximal principle. Qiu \cite{qiu2019interior} \cite{qiu2017interiorHessian} removed the convexity condition in dimension 3. Integral approach for equation (\ref{eq:sigma_lambda}) with $f\equiv 1$ was employed by Warren-Yuan \cite{warren2009hessian} for dimension $3$ and Shankar-Yuan \cite{shankar2020hessian}  and for general dimensions but with semi-convexity assumption on the  solutions respectively. Their basic idea is that mean value property of $\lambda_{\max}$, the maximum eigenvalue of $D^2u$,  can be derived from Jacobi inequality.
Recently, Shankar-Yuan \cite{shankar2023hessian} used a doubling argument to obtain the Hessian estimate for equation (\ref{eq:sigma_lambda}) in dimension 4   without any assumption on the convexity of the solution.
Compact method for equation (\ref{eq:sigma_lambda}) in general dimensions was used in McGonagle-Song-Yuan \cite{MR3913193}. We also mention the global Hessian and gradient estimates for these two equations, see the works of Caffarelli-Nirenberg-Spruck \cite{MR0806416} \cite{Caffarelli1986starshape}, Trudinger \cite{trudinger1990dirichlet} \cite{MR1368245}, Lin-Trudinger \cite{MR1281990}, Ivochkina-Lin-Trudinger \cite{MR1449405}, Guan-Ren-Wang \cite{MR3366747}, and Ren-Wang \cite{MR4011801}.

However, all those Hessian estimates appeared in the papers we have known up to now depend strongly on the quantity  $\|f\|_{C^k}$ for some $k\geq 2$. Hence, a natural question  is:
\begin{itemize}
    \item  [\textbf{II}]  What is the minimum number $k$ in Question {\bf I} ?
\end{itemize}

Notice that when $n=2$, equations (\ref{eq:sigma_kappa}) and (\ref{eq:sigma_lambda}) are reduced to Monge-Amp\`ere equations and the solutions are convex. By the foundational regularity results  due to Caffarelli \cite{c}, one may ask:
\begin{itemize}
 \item[\textbf{III}] Fix $\alpha\in[0, 1]$. For strictly convex  smooth solutions to  equations (\ref{eq:sigma_kappa}) and (\ref{eq:sigma_lambda}),   can we have the interior a priori $C^{2,\alpha}$ estimates
    \begin{eqnarray*}
    \|D^2u\|_{C^{\alpha}(B_{\frac12}(0))}\le C(n, \|u\|_{C^1(B_1(0))},\|f^{-1}\|_{L^\infty(B_1(0))},\|f\|_{C^\alpha(B_1(0))})?
    \end{eqnarray*}
    \end{itemize}
This estimate for  Monge-Amp\`ere equations  was proved  by Caffarelli in \cite{c} for $0<\alpha<1$ and by Jian-Wang in \cite{jw} for $\alpha\in[0, 1]$.  Hence Question \textbf{III}    is correct in the two dimensional case ($n=2$).  By the way, we would like to mention that the global a priori $C^{2,\alpha}$ estimates for  Monge-Amp\`ere equations with
Dirichlet boundary value, natural boundary value and oblique boundary value were obtained in \cite{s,tw}, \cite{clw} and \cite{jt}, respectively.

In this paper, we aim to take one step towards Problems \textbf{II} and \textbf{III}.
Specifically, we explore the possibility of obtaining interior a priori Hessian estimates which depends on  $\|f\|_{C^k}$  for some $k$ less than $2$ but larger than  $\alpha$.

The following two theorems are the  main results of this paper.

\begin{theorem}\label{Thm_kappa}
If $M=\{(x, u(x)): x\in B_{2}(0)\}$ is a smooth convex solution to equation (\ref{eq:sigma_kappa})   with $\inf_{x\in B_{2}(0)}f(x)>0$, then we have
\begin{eqnarray*}
|\kappa(0)|\le C(n,\|u\|_{C^{0,1}(B_{2}(0))}, \|f^{-1}\|_{L^{\infty}(B_{2}(0))}, \|f\|_{C^{0,1}(B_{2}(0))}).
\end{eqnarray*}
\end{theorem}

\begin{theorem}\label{Thm_sigma}
 If $u$ is a smooth convex solution to equation (\ref{eq:sigma_lambda}) in $B_{2}(0)$ with $\inf_{x\in B_{2}(0)}f(x)>0$, then we have
\begin{eqnarray*}
|D^2u(0)|\le C(n,\|u\|_{C^{0,1}(B_{2}(0))}, \|f^{-1}\|_{L^{\infty}(B_{2}(0))}, \|f\|_{C^{0,1}(B_{2}(0))}).
\end{eqnarray*}
\end{theorem}

\
We  remark that the most usual approach  to get the interior a priori Hessian estimates for  (\ref{eq:sigma_kappa}) and (\ref{eq:sigma_lambda}) is pointwise argument, which  requires to differentiate the equations twice, then  construct  auxiliary functions and apply the maximal principle.  So  $\|f\|_{ C^2}$ is necessary for these approaches. When $f\equiv 1$,
Shankar-Yuan \cite{shankar2020hessian} proved a mean value inequality for equation $\sigma_2=1$, thus providing an integral approach to Hessian estimates.

We recall the method of integrating by parts in linear elliptic equations, and there is a well known technique called Moser iteration. Ideally, we may differentiate the equation as many times as we want, then integrate by parts to eliminate the singular derivatives. When we try to apply this technique to the fully nonlinear equation $\sigma_2=f$,  two challenges appear, which bring us many difficulties. The first is  to seek for a  divergence structure to initiate the integration by parts, and the second, more challenging, is that the linearized operator of $\sigma_2$ may be degenerate. Such degeneracy will cause the Moser iteration to be unable to continue, as shown in the works of Urbas \cite{MR1777141,MR1840289}.

We illustrate the idea of our new approach to overcome these difficulties. The recent work of Shankar-Yuan \cite{shankar2023hessian} has showed that the mean curvature or $\Delta u$ is in fact strong subharmonic (Jacobi inequality). This gives an ideal divergence structure to apply the method of integration by parts. However, unless we address the degeneracy of the $\sigma_2$ operator, our approach will not essentially differ from the work of Urbas.
We make an important observation that, during integration by parts, such degeneracy appears only on the boundary, so our strategy is  to cut the boundary off.

Our new argument, which we call boundary Jacobi inequality approach, involves two cut-off functions.
The first is to initiate the Moser iteration, which leads to $C^{1,1}$ estimates in terms of integral of mean curvature or $\Delta u$ to some power $p=p(n)$.
The second gives us another iteration process to bound the above integral through integration by parts.
With each steps, the power increases, ultimately leading to $W^{2,p}$ estimates for any prescribed integer $p>1$.
In detail, the second cut-off function is the composition of a monotonic function and a truncation function supported in $B_1(0)$.
This truncation function is obtained through cutting the surface $M$ with a 2-convex function, which departs from $M$ at the origin by a uniform distance.
When the solution is convex, such 2-convex function exists in a long thin strip contained in $B_1(0)$.
The 2-convexity of the cutting function is essential, since the level set of the solution can be degenerate due to the unknown boundary values.
For example, consider the function $v(x,y,z) = x^2+y^2$, whose level set $\{v=0\}$ degenerates to a line crossing the origin. We also refer to the works of Chou-Wang \cite{MR1835381} and Mooney \cite{MR4246798}.

In all, the second cut-off function enables us to estimate the integral of certain power of mean curvature or $\Delta u$ in some irregular domains contained in $B_1(0)$.
Generally, such integral cannot be estimated in $B_1(0)$ since the $\sigma_2$ operator may be degenerate.
This cut-off function, together with the Jacobi inequality, helps us to eliminate the sigularity at the boundary.
We point out that this argument can proceed only when the solution is convex. Otherwise the Jacobi inequality may fail, as showed in Shankar-Yuan \cite{shankar2023hessian}.

The paper is organized as follows. In Section 2 we first review some standard facts of moving frames, then develop a boundary Jacobi inequality, which plays a key role in our argument, and finally  use it to prove Theorem \ref{Thm_kappa}. In Section 3 we prove Theorem \ref{Thm_sigma},  the case of Hessian equation. Although the argument is similar to that of Theorem \ref{Thm_kappa}, we still provide an elaborate proof for the convenience of the readers who   are more concerned about the Euclidean case.

\section{ A boundary Jacobi inequality approach}

In this section, we will first prove a boundary Jacobi inequality for convex solution to scalar curvature equation (\ref{eq:sigma_kappa}). Then the Pogorolov $W^{2,p}$ and $C^{1,1}$ estimates follow from this inequality and by integration method. Let us first introduce some  notations and basic formulas of moving frames that will be used from time to time without any citation in this section.

\

We choose an orthonormal frame $\{e_{1},e_{2},\cdots,e_{n}, \nu\}$ in $\mathbb{R}^{n+1}$ such that $e_{1},e_{2},\cdots,e_{n}$ are tangent to $M$ and $\nu$ is the outer normal on $M$. Denote the dual form and the connection form by $\omega^i$ and $\omega^j_i$, where $i$ and $j$ range from $1$ to $n$. The second fundamental form is denoted by $h_{ij}$.
Recall the following fundamental formulas:
\begin{eqnarray*}
X_{ij} & = & -h_{ij}\nu\begin{array}{cc}
 & (Gauss\,\,formula)\end{array}\\
 \nu_i & = & h_{ij}e_j \begin{array}{cc}
 & (Weingarten\,\,equation)\end{array}\\
h_{ijk} & = & h_{ikj}\begin{array}{cc}
 & (Codazzi\,\,equation)\end{array}\\
R_{ijkl} & = & h_{ik}h_{jl}-h_{il}h_{jk}\begin{array}{cc}
 & (Gauss\,\,equation)\end{array}.
\end{eqnarray*}
We also have the following commutator formula: 
\begin{eqnarray*}
h_{ijkl}-h_{ijlk} & = & \sum_m h_{im}R_{mjkl}+\sum_m h_{mj}R_{mikl}.
\end{eqnarray*}
Combining Codazzi equation, Gauss equation and the commutator formula, we have 
\begin{eqnarray}
h_{iikk}=h_{kkii}+h_{kk}\sum_{m}h_{im}^2-h_{ii} \sum_{m} h_{mk}^{2}.\label{eq:commute}
\end{eqnarray}
The mean curvature of $M$ is defined by 
\begin{eqnarray*}
H=\sum_{i}h_{ii}.
\end{eqnarray*}
We choose the outer normal $\nu=\frac{1}{W}(Du,-1)$, where $W=\sqrt{1+|Du|^2}$. Since the position vector of $M$ is $X=(x, u(x))$, we will view $u$  as a $C^2$ function on $M$ by $u =  \langle X,E_{n+1} \rangle$, then
\begin{eqnarray*}
u_i&=&\langle e_i,E_{n+1}\rangle,\\
u_{ij}&=&-h_{ij}\langle \nu,E_{n+1}\rangle =\frac{h_{ij}}{W}.
\end{eqnarray*}
Here and below, $E_{k}$ denote the unit vector in $\mathbb R^{n+1}$, the $k$-th component of which is $1$, and the others are all zero, $1\leq k\leq n+1$. $u_i, u_{ij}$ are the covariant derivatives of $u$ with respect to the moving frame. So the mean curvature can be expressed by
\begin{eqnarray*}
H=W\sum_iu_{ii}=W\Delta  u.
\end{eqnarray*}
From the above qualities, we get an important relation that 
$$
|\nabla u|^2={\sum^n_{i=1}\langle e_i,E_{n+1}\rangle ^2}\le 1.
$$
Generally, we can treat a function $v$ in $\mathbb R^n$ as a function on $M$ by defining 
$$
v(X)=v(<X,E_1>,<X,E_2>,\cdots,<X,E_n>).
$$
This gives $v_i = \sum^n_{l=1}\partial_lv<e_i,E_l>$, and therefore
$$
|\nabla v|^2= \sum^n_{i,k,l=1}\partial_k v\,\partial_l v \,\langle e_i,E_k\rangle \langle e_i,E_l \rangle \le |Dv|^2.
$$
There are also cases when local frame $\{\partial_1,\cdots\partial_n,\nu\}$ is needed. The metric $g:=[g_{ij}]_{n\times n}$ is
$$
g_{ij}=<\partial_i,\partial_j>=\delta_{ij}+u_iu_j,
$$
and the inverse of the metric $g^{-1}:=[g^{ij}]_{n\times n}$ is given by
$$
g^{ij}=\delta_{ij}-\frac{u_iu_j}{W^2}.
$$
Using  \textit{II}$:=[h_i^j]_{n\times n}$ to denote the second fundamental form, we have
$$
h_i^j=\partial_i\left({\frac{u_j}{W}}\right), \ \ h_{ij}=\frac{u_{ij}}{W}.
$$
Finally, in local coordinates, let
$$
F(D^2u,Du):=\sigma_2(\kappa_1,\kappa_2,\cdots, \kappa_n),
$$
and denote $F^{ij}=\frac{\partial F}{\partial h_{ij}}$, then we have 
$$ F^{ij} = H g^{ij} - g^{is} h_{sk} g^{kj} = H g^{ij} - h^i_k g^{kj}, $$
that is 
\begin{equation} \label{eq:Fij}
    [F^{ij}]_{n\times n} = H g^{-1} - \textit{II}\, g^{-1}.
\end{equation}
See, for example, Guan-Qiu \cite{guan2019interior}.

The following notations will also be used throughout this section, 
$$
\Delta_F v :=F^{ij}v_{ij}, \ |\nabla_F v|^2 :=F^{ij}v_iv_j.
$$
One can verify that the linearized operator $\Delta_F$  has a divergence structure similar to the Laplace-Beltrami operator, which allows us to integrate by parts.

From now on to the end of this section, we assume $u$ and $f$ are the same as in Theorem \ref{Thm_kappa}. Note that $f$ is smooth in $B_2(0)$ by the smoothness of $u$ in equation (\ref{eq:sigma_kappa}).

\begin{lemma}\label{lemma:curvature_Trace_Jacobi}
    For each constant $\epsilon\in(0,1]$, we have
\begin{equation}\label{eq:curvature_Trace_Jacobi}
\Delta_FH-(2-\epsilon)H^{-1}|\nabla_FH|^2\ge \Delta f - C(\epsilon)f^{-1}|\nabla f|^2.
\end{equation}
\end{lemma}
\begin{proof}
Differentiate equation (\ref{eq:sigma_kappa}) twice, we get
\begin{gather*}
F^{ij}h_{ijk} = f_k,\\
\sum^{n}_{i,j,s,t=1}(\delta_{ij}\delta_{st}-\delta_{it}\delta_{js})h_{ijk}h_{stk}+F^{ij}h_{ijkk} = f_{kk}.
\end{gather*}
Summing up the second equation from $k = 1$ to $n$, we have 
\begin{eqnarray*}
    \sum^n_{k=1}F^{ij}h_{ijkk} &=& \sum^n_{i,j,k=1}h_{ijk}^2-\sum^n_{k=1}(\sum^n_{i=1}h_{iik})^2+\Delta f
    \\
    &=& 3\sum^n_{i\neq k}h^2_{iik}+\sum^n_{k=1}h^2_{kkk}+\sum^n_{i\neq j\neq k}h^2_{ijk}-\sum^n_{k=1}(\sum^n_{i=1}h_{iik})^2+\Delta f.
\end{eqnarray*}
Choose an orthogonal frame such that the second fundamental form is diagonal at a point $p\in M$. We use equality (\ref{eq:commute}) and get
\begin{eqnarray*}
    \Delta_F H\ge
    3\sum^n_{i\neq k}h^2_{iik}+\sum^n_{k=1}h^2_{kkk} -\sum^n_{k=1}(\sum^n_{i=1}h_{iik})^2+\Delta f+ F^{ii}(\kappa_i\sum^n_{k=1}\kappa_k^2-\kappa_i^2H).
\end{eqnarray*}
Let $\delta \ge 0$ be a constant, we have $F^{ij}= \delta_{ij}(H-\kappa_i)$ at $p$, then
\begin{align*}
& \Delta_F H -\delta H^{-1}{|\nabla_F H|^2}\\
\ge \ & 3\sum^n_{i\neq k}h^2_{iik}+\sum^n_{k=1}h^2_{kkk}-\sum^n_{k=1}(1+\delta\frac{ H-\kappa_k}{H})(\sum^n_{i=1}h_{iik})^2+\Delta f+F^{ii}(\kappa_i\sum^n_{k=1}\kappa_k^2-\kappa_i^2H).
\end{align*}
Since we have $F^{ij}h_{ijk}=f_k$, denote
\begin{eqnarray*}
Q_k & := & 3\sum^n_{i\neq k}h^2_{iik}+h^2_{kkk}-(1+\delta\frac{H-\kappa_k}{H})(\sum^n_{i=1}h_{iik})^2 \\
& = & 3\sum^n_{i\neq k}h^2_{iik}+h^2_{kkk}-(1+\delta\frac{H-\kappa_k}{H})[\sum^n_{i=1}(1-\frac{H-\kappa_i}{H})h_{iik}+\frac {f_k}{H}]^2.
\end{eqnarray*}
To estimate the quadratic form $Q_k$, we introduce the following lemma (for the proof we refer to Shankar-Yuan \cite{SY22} and Zhou \cite{zhou2023hessian}).
Let \{$a_1,\cdots a_n$\} be positive constants, $n\ge1$, $L\in \mathbb{R}^n$. Consider the following $n\times n$ symmetric matrix
$$
\Lambda=\sum^{n}_{i=1}a_ie_i^Te_i-L^TL.
$$
It is proved that $\Lambda \succeq 0$ is equivalent to 
$$
1-\sum^{n}_{i=1}\frac{1}{a_i}|\langle L,e_i \rangle|^2\ge 0.
$$
Using this lemma, we estimate the following quadratic 
\begin{eqnarray*}
    \hat Q_k&=&1-\frac1{3}(1+\delta\frac{H-\kappa_k}{H})\sum_{i\neq k}(1-\frac{H-\kappa_i}{H})^2 - (1+\delta\frac{H-\kappa_k}{H})(1-\frac{H-\kappa_k}{H})^2\\
    &\ge& 1-\frac1{H^2}[\frac1{3}(1+\delta)\sum^n_{i\neq k}\kappa_i^2+\kappa_k^2+\delta\frac{\kappa^2_k(H-\kappa_k)}{H}]\\
    &\ge& 1-\frac1{H^2}[\frac1{3}(1+\delta)\sum^n_{i\neq k}\kappa_i^2+\kappa_i^2+\delta f].
\end{eqnarray*}
Clearly $\hat Q_k\ge0$ if we choose $\delta=2-\frac\epsilon2$, since $H^2= \sum^n_{i=1}\kappa_i^2+2f$. We  get
\begin{eqnarray*}
Q_k&\ge& -2(1+\delta\frac{H-\kappa_k}{H})[\sum^n_{i=1}(1-\frac{H-\kappa_i}{H})h_{iik}]\frac{f_k}H-(1+\delta\frac{H-\kappa_k}{H})(\frac{f_k}H)^2 \\
& = & -2(1+\delta\frac{H-\kappa_k}{H})(\sum^n_{i=1}h_{iik})\frac{f_k}{H} + (1+\delta\frac{H-\kappa_k}{H})(\frac{f_k}H)^2 \\
&\ge & -6|\sum^n_{i=1}h_{iik}|\cdot|\frac{f_k}H| = -6 \frac{|H_k \cdot f_k|}{H}.
\end{eqnarray*}
Since $H\ge \sqrt{f}>0$, we choose a small constant $\hat\epsilon=\frac{\epsilon f}{2}$,
\begin{equation}\label{eq_Q1_kappa}
Q_k\ge -\frac{\hat\epsilon}{H^2}H_k^2-\frac{C(1)}{\hat\epsilon}|\nabla f|^2.
\end{equation}
Combining inequality (\ref{eq_Q1_kappa}) and the relation $\frac{H-\kappa_k}{H}\ge \frac{f}{H^2}$ gives
\begin{eqnarray*}
\Delta_FH -(2-\epsilon)H^{-1}{|\nabla_F H|^2}\ge \Delta f- C(\epsilon)f^{-1}|\nabla f|^2+F^{ii}(\kappa_i\sum^n_{k=1}\kappa_k^2-\kappa_i^2H).
\end{eqnarray*}
As for the commutator terms, we have
\begin{eqnarray*}
F^{ii}(\kappa_i\sum^n_{k=1}\kappa_k^2-\kappa_i^2H) & = & 2 \sigma_2 \sum^n_{i=1}\kappa_i^2- H\sum_{i=1}^n (H-\kappa_i) \kappa_i^2 \\
& = & 2 \sigma_2 \sum^n_{i=1}\kappa_i^2- H^2\sum_{i=1}^n \kappa_i^2 + H \sum_{i=1}^n \kappa_i^3 \\
& = & (2 \sigma_2 - H^2)\sum_{i=1}^n \kappa_i^2 + H \sum_{i=1}^n \kappa_i^3 \\
& = & -(\sum_{i=1}^n \kappa_i^2)^2 + (\sum_{i=1}^n \kappa_i)(\sum_{i=1}^n \kappa_i^3) \geq 0.
\end{eqnarray*}
Thus we get the trace Jacobi inequality for the mean curvature $H$,
\begin{eqnarray*}
\Delta_FH -(2-\epsilon)H^{-1}{|\nabla_F H|^2}\ge \Delta  f- C(\epsilon)f^{-1}|\nabla f|^2.
\end{eqnarray*}
\end{proof}

\begin{lemma}[Boundary Jacobi inequality]\label{lemma:curvature_Boundary_Jacobi}
Let $u$ be a smooth convex solution to equation (\ref{eq:sigma_kappa}) in $B_2(0)$.
Let $\varphi(t)=(t^+)^4$, $t\in \mathbb{R}$ be a $C^2$ function. Suppose that $f$ is smooth. Then for each smooth function $w(x)$, 
\begin{align*}
\Delta_F [\varphi(w-u)H ]
&\ge [\varphi'(w-u)\Delta_F(w-u)]H+ [{\Delta f}-C(1)f^{-1}|\nabla f|^2]\varphi(w-u).
\end{align*}

\end{lemma}
\begin{proof}
We begin by direct calculation at the points where $w-u>0$,
\begin{align*}
&\Delta_F[\varphi(w-u)]=\varphi''(w-u)|\nabla_F (w-u)|^2+\varphi'(w-u)\Delta_F(w-u),\\
&|\nabla_F[\varphi(w-u)]|^2=[\varphi'(w-u)]^2|\nabla_F (w-u)|^2.
\end{align*}
Then we get
\begin{align*}
&\Delta_F[\varphi(w-u)]-\frac23|\nabla_F[\varphi(w-u)]|^2[\varphi(w-u)]^{-1}\\
&=\{\varphi''(w-u)-\frac23[\varphi(w-u)]^{-1}[\varphi'(w-u)]^2\}|\nabla_F (w-u)|^2+\varphi'(w-u)\Delta_F(w-u)\\
&\ge \varphi'(w-u)\Delta_F(w-u).
\end{align*}
We also know from Lemma \ref{lemma:curvature_Trace_Jacobi} (choose $\epsilon=\frac12$) that
$$
\Delta_FH -\frac{3}{2}H^{-1}{|\nabla_F H|^2}\ge \Delta f- C(1)f^{-1}|\nabla f|^2.
$$
Now we consider the boundary Jacobi inequality for the function $\varphi(w-u)H$ which vanishes at the level set $\{w=u\}$,
\begin{align*}
\Delta_F(\varphi H)
&= H\Delta_F\varphi+\varphi\Delta_FH+2F^{ij}\varphi_iH_j
\\
&\ge\{\Delta_F\varphi-\frac23\varphi^{-1}|\nabla_F\varphi|^2\}H
+\{\Delta_FH -\frac32H^{-1}{|\nabla_F H|^2}\}\varphi\\
&\ge [\varphi'\Delta_F(w-u)]H+[\Delta f- C(1)f^{-1}|\nabla f|^2]\varphi.
\end{align*} 
\end{proof}

\subsection{Pogorelov-type \texorpdfstring{$W^{2,p}$}{} estimates}
Now we use the boundary Jacobi inequality (Lemma \ref{lemma:curvature_Boundary_Jacobi}) to obtain a priori $W^{2,p}$ estimates for  equation (\ref{eq:sigma_kappa}).
We first prove the following lemma that concerns  2-convex functions.

\begin{definition}
We say that a $C^2$ function $w$ on $M$ is $2$-convex, if 
$$ \sigma_k(g^{-1}D^2w)>0,\ \ k=1,2. $$
\end{definition}

\begin{lemma}\label{lemma:2-convex_on_M}
    Suppose $w$  is 2-convex on $M$. Then 
    $$
    \Delta_F w\ge -C(n,\|w\|_{C^{0,1}(M)},\|u\|_{C^{0,1}(M)},\|f\|_{L^\infty(M)}).
    $$
\end{lemma}
\begin{proof}
We introduce the local coordinates,
\begin{eqnarray*}
 w_{ij}=D_{ij}w-g^{kl}u_kw_lD_{ij}u.
\end{eqnarray*} 
Using $\sigma_2(g^{-1}D^2w)>0$ and equation \eqref{eq:Fij}, we have
\begin{eqnarray*}
\Delta_F w &=& (H g^{ij} - h^i_k g^{kj}) D_{ij}w - \frac{\partial \sigma_2}{\partial h_{ij}} g^{kl} u_k w_l u_{ij} \\
&=& (H g^{ij} - h^i_k g^{kj}) D_{ij}w - \frac{\partial \sigma_2}{\partial h_{ij}}  h_{ij} W g^{kl} u_k w_l  \\
&=& (H g^{ij} - h^i_k g^{kj}) D_{ij}w - \frac{\partial \sigma_2}{\partial h_i^j}h_i^j W g^{kl} u_k w_l \\
&=&(H g^{ij} - h^i_k g^{kj}) D_{ij}w - 2 f W g^{kl} u_k w_l\\
&=&\mathrm{Trace}(\textit{II}\, )\,\mathrm{Trace}(g^{-1}D^2w)-\mathrm{Trace}(\textit{II} \, g^{-1}D^2w)-2fWg^{kl}u_kw_l\\
&\ge &-C(n,\|w\|_{C^{0,1}(M)},\|u\|_{C^{0,1}(M)},\|f\|_{L^\infty(M)}).
\end{eqnarray*} 
\end{proof}

\begin{prop}\label{lemma:curvature_W2p_estimate}
Suppose $u$ is a smooth convex solution to equation (\ref{eq:sigma_kappa}) in $  B_2(0)$. Let $w$ be a $2$-convex function on the graph of $u$, and $\Omega \subset B_1(0)$ satisfies $w\le u$ on $\partial \Omega$. Define $\varphi(t)$ as in Lemma \ref{lemma:curvature_Boundary_Jacobi}. Suppose that $f>0$ in $ \overline{B_2(0)}$, $f$ is smooth. Then for $p=1,2,\cdots$,
$$
    \int_{\Omega}[\varphi(w-u)]^{p-1}H^p dM\le p!\, C(n,\|w\|_{C^{0,1}(B_2(0))},\|u\|_{C^{0,1}(B_2(0))},\|f^{-1}\|_{L^{\infty}(B_2(0))},\|f\|_{C^{0,1}(B_2(0))}).
$$
\end{prop}
\begin{proof}
Let $p\ge 1$ be an integer. We choose $[\varphi(w-u)H]^p$ as a test function. By Lemma \ref{lemma:curvature_Boundary_Jacobi}, we have
\begin{equation}\label{eq:temp_1}
\begin{aligned}
 &p\int_{\Omega}|\nabla_F(\varphi H)|^2(\varphi H)^{p-1} dM\\
 &\le -\int_{\Omega}[\varphi' \Delta_F (w-u)]H^{p+1}\varphi ^p dM
 -\int_{\Omega}[\Delta f- C(1)f^{-1}|\nabla
  f|^2]H^p\varphi ^{p+1} dM.
\end{aligned}
\end{equation}
Let $0<\delta<1$, $C=1+\|Du\|^2_{L^\infty(\Omega)}$, we integrate by parts to estimate that
\begin{eqnarray*}
    \int_{\Omega}H^{p+1}\varphi ^pdM&=&\int_{\Omega}(\varphi H)^{p}\Delta u W dM
    \le -C\int_{\Omega}\nabla [(\varphi H)^{p}]\nabla u dM\\
    &\le& \delta p\int_{\Omega} [\nabla(\varphi H)]^2 H^{p-2}\varphi^{p-1} dM
    +\frac {pC^2}{\delta} \int_{\Omega}H^{p}\varphi ^{p-1}dM,\\
    \int_{\Omega}H^{p}\varphi ^{p+1}\Delta fdM&=&
-\int_{\Omega}\varphi \nabla f\nabla([\varphi H]^{p}) dM-\int_{\Omega}(\varphi H)^{p}\nabla f\nabla\varphi  dM\\
&\le& \delta p\int_{\Omega} |\nabla(\varphi H)|^2 H^{p-2}\varphi^{p-1} dM
    +\frac {pC_1}\delta  \int_{\Omega}H^{p}\varphi ^{p-1}dM,
\end{eqnarray*}
where $C_1=C(n,\|\varphi\|_{C^{0,1}(\Omega)}, \|Df\|_{L^\infty(\Omega)})$. Recall $(H-\kappa_i)H\ge f>0$ for each $1\le i\le n$, we have
$$
|\nabla_F(\varphi H)|^2(\varphi H)^{p-1}\ge f|\nabla(\varphi H)|^2 H^{p-2}\varphi^{p-1} .
$$
 For the first term at the right hand side of (\ref{eq:temp_1}), recall $\sigma_2(\kappa)=f$, we have
\begin{eqnarray*}
\Delta_F u =  \frac{2f}{W} \le 2\|f\|_{L^\infty(\Omega)}.
\end{eqnarray*}
Using Lemma \ref{lemma:2-convex_on_M}, we get 
$$
\varphi'\Delta_F(w-u)\ge -C(n,\|w\|_{C^{0,1}(\Omega)},\|u\|_{C^{0,1}(\Omega)},\|f\|_{L^\infty(\Omega)}):=-C_2.
$$ 
Let $\delta=\frac1{4(C_2+1)}$, using (\ref{eq:temp_1}) we get the recursion formula for $p\ge 1$,
$$
\int_{\Omega}H^{p+1}\varphi ^pdM\le pC_*\int_{\Omega}H^{p}\varphi ^{p-1}dM,
$$
where $C_*=C_*(C,C_1,C_2,\|f^{-1}\|_{L^\infty(\Omega)})$. 
When $p=1$, we choose a cutoff function $\phi \in C_0^{\infty}(B_2(0)),\phi =1 $ in $B_1(0)$, then 
\begin{eqnarray*}
    \int_\Omega H dM & \leq & C \int_\Omega H dx = C \int_\Omega div(\frac{Du}{W}) dx \\
    & \leq & C \int_{B_2} \phi^2 div(\frac{Du}{W}) dx \leq C \int_{B_2} \frac{D\phi \cdot Du}{W} dx \\
    &\leq & C(\|Du\|_{L^{\infty}(B_2(0)}).
\end{eqnarray*}

\end{proof}

\subsection{Pogorelov-type \texorpdfstring{$C^{1,1}$}{} estimates}
As long as $H$ is subharmonic, we can extend the results in Urbas \cite{MR1777141}\cite{MR1840289} to Lipschitz variable right hand side $f(x)$, and prove a $C^{1,1}$ bound by the $W^{2,p}$ norm of $u$ for equation (\ref{eq:sigma_kappa}).
Still it hardly leads to a complete $C^{1,1}$ estimate, since the section of the level set is not necessarily normalized in Proposition \ref{lemma:curvature_W2p_estimate}. Here we prove a Pogorelov-type $C^{1,1}$ estimate that will match Proposition \ref{lemma:curvature_W2p_estimate}.

\begin{prop}\label{lemma:curvarure_C11ptype_estimate}
Let $w$ be a smooth function  on $M$. Let $u$ be a smooth convex solution to equation (\ref{eq:sigma_kappa}) in $ B_2(0)$. Define $\varphi(t)$ as in Lemma \ref{lemma:curvature_Boundary_Jacobi}. Suppose that $f>0$ in $ \overline {B_2(0)}$, $f$ is smooth. Then
$$
    \|[\varphi(w-u)]^{2n}H\|_{L^{\infty} (B_{\frac{1}{2}}(0))}\le C \int_{B_{1}(0)}\varphi^{n} H^{n+1}dx.
$$
\end{prop}
where $C=C(n,\|u\|_{C^{0,1}(B_1(0))},\|w\|_{C^{0,1}(B_1(0))},\|f\|_{C^{0,1}(B_1(0))},\|f^{-1}\|_{L^{\infty}(B_1(0))})$.
\begin{proof}
Let $0<\rho<\frac12$, choose a cut-off function $\phi\in C^\infty_0(B_{\frac{1}{2}+\rho}(0))$ with $\phi=1$ in $B_{\frac{1}{2}}(0)$, $\phi\ge 0$, $|D\phi|\le 8\rho^{-1} $. Let $p\ge 1,q\ge 2$ be a pair of positive constants satisfying $\frac{q}{p}\le 2n $. According to Lemma \ref{lemma:curvature_Trace_Jacobi}, $H$ is subharmonic, we have
$$
\int_{B_1(0)}H^p[\varphi(w-u)]^{q}\phi^2\Delta_FH dM\ge\int_{B_1(0)}[{\Delta f}-C(1)f^{-1}|\nabla f|^2]H^p[\varphi(w-u)]^{q}\phi^2 dM.
$$
Recall $(H-\kappa_i)H\ge f$ for all $1\le i\le n$, we have $$|\nabla_F H|^2H\ge f|\nabla H|^2\ge C(\|Du\|_{L^\infty(B_1(0))}) f|DH|^2.$$
Integrating by parts the above inequality and using $H\ge \sqrt{f}>0$, we get
\begin{align*}
\int_{B_1(0)}|D H|^2H^{p-2}\varphi^q\phi^2dM
&\le C_1\int_{B_{{\frac12}+\rho}(0)}(|\nabla\phi|^2+|\nabla f|^2+|\nabla \varphi|^2)H^{p+2}\varphi^{q-2}dM \\
& \le \frac{C_2}{\rho^2} \int_{B_{{\frac12}+\rho}(0)} H^{p+2}\varphi^{q-2} dx,
\end{align*}
where $C_1=C(\|u\|_{C^{0,1}(B_1)},\|\varphi\|_{L^\infty(B_1)},\|f^{-1}\|_{L^\infty(B_1)})$, $C_2= C_1\cdot( 1+\|f\|^2_{C^{0,1}(B_1)}+\|\varphi\|^2_{C^{0,1}(B_1)})$. 
Then we have
\begin{align*}
\int_{B_{\frac{1}{2}}(0)}|D[H^{\frac p2}\varphi^{\frac q2}]|^2dx
\le \frac{C_2p^2q^2}{\rho^2}\int_{B_{{\frac{1}{2}}+\rho}(0)}H^{p+2}\varphi^{q-2}dx,
\end{align*}
and
\begin{align*}
\int_{B_{\frac{1}{2}}(0)}|H^{\frac p2}\varphi^{\frac q2}|^2 dx
\le C_1\int_{B_{{\frac{1}{2}}+\rho}(0)}H^{p+2}\varphi^{q-2}dx \le \frac{C_2p^2q^2}{\rho^2}\int_{B_{{\frac{1}{2}}+\rho}(0)}H^{p+2}\varphi^{q-2}dx .
\end{align*}
For $n>2$, we get an iteration formula by Sobolev embedding theorem
\begin{align*}
\int_{B_{\frac{1}{2}}(0)}H^{\frac {np}{n-2}}\varphi^{\frac {nq}{n-2}} dx
& \le [\frac{C_2p^2q^2}{\rho^2}\int_{B_{{\frac{1}{2}}+\rho}(0)}H^{p+2}\varphi^{q-2}dx]^{\frac {n}{n-2}} \\
& \le [\frac{C_3 p^4}{\rho^2}\int_{B_{{\frac{1}{2}}+\rho}(0)}H^{p+2}\varphi^{q-2}dx]^{\frac {n}{n-2}},
\end{align*}
where $ C_3 =  4n^2 C_2 $, since we have required $\frac{q}{p} \le 2n $.
Denote $\gamma=\frac{n}{n-2}$, we fix an integer $k_0\ge\frac{\ln n}{\ln\gamma}$, and take $p_0 = \gamma^{k_0}$, $q_0 = 2np_0$ to initiate the iteration. For $k = 1,\cdots,k_0$, let
\begin{eqnarray*}
    p_{k} &=& \gamma^{-1} p_{k-1}+2=\gamma^{-k}p_0+2\sum^{k-1}_{i=0}\gamma^{-i}, \\
    q_{k} &=& \gamma^{-1} q_{k-1}-2=\gamma^{-k}q_0-2\sum^{k-1}_{i=0}\gamma^{-i}, \\
    r_k &=& \frac{1}{2}+\sum^{k}_{i=1}2^{-(k_0-i+2)}.
\end{eqnarray*}
We can check that $p_k > n$ and $q_k \ge 2 $ satisfy $\frac{q_k}{p_k} \le 2n$ all along.
Then we can rewrite the formula as
$$
\int_{B_{r_{k-1}}}H^{p_{k-1}}\varphi^{q_{k-1}}dx 
\le [C_3\, 2^{2(k_0-k+2)} (p_k-2)^4 \int_{B_{r_{k}}}H^{p_k}\varphi^{q_{k}}dx]^{\gamma}.
$$
Therefore, an iteration process implies
\begin{align*}
\int_{B_{\frac{1}{2}}(0)}H^{p_0}\varphi^{q_{0}}dx
&\le C_3^{\sum^{k_0}_{k=1}\gamma^k}2^{2\sum^{k_0}_{k=1}(k_0-k+2)\gamma^k}\prod ^{k_0}_{k=1}(p_k-2)^{4\gamma^k}[\int_{B_{r_{k_0}}(0)}H^{p_{k_0}}\varphi^{q_{k_0}}dx]^{\gamma^{k_0}}\\
&\le C_3^{\sum^{k_0}_{k=1}\gamma^k} 4^{\sum^{k_0}_{k=1}(k_0-k+2)\gamma^k} \prod^{k_0}_{k=1}(n\gamma^{-k}p_0)^{4\gamma^k} [\int_{B_{r_{k_0}}(0)}H^{p_{k_0}}\varphi^{q_{k_0}}dx]^{\gamma^{k_0}} \\
& = (n^4 C_3)^{\sum^{k_0}_{k=1}\gamma^k} 4^{\sum^{k_0}_{k=1}(k_0-k+2)\gamma^k} \gamma^{4\sum^{k_0}_{k=1}(k_0-k)\gamma^k} [\int_{B_{r_{k_0}}(0)}H^{p_{k_0}}\varphi^{q_{k_0}}dx]^{p_0}.
\end{align*}

Notice that $\frac{1}{\gamma^{k_0}} \sum^{k_0}_{k=1}\gamma^k $ and $ \frac{1}{\gamma^{k_0}} \sum^{k_0}_{k=1}(k_0-k)\gamma^k $ can both be bounded by some constant $C(n)$. So we get
\begin{align*}
\|\varphi^{2n}H\|_{L^{p_0} (B_{\frac{1}{2}}(0))} = [\int_{B_{\frac{1}{2}}(0)}H^{p_0}\varphi^{q_{0}}dx]^\frac 1{p_0}
&\le C(n)\cdot C_3^{C(n)} \int_{B_{1}(0)}H^{p_{k_0}}\varphi^{q_{k_0}}dx.
\end{align*}
From the choice of $p_0, q_0 $ and $k_0$, we can verify that $p_{k_0} \le n+1 $ and $q_{k_0} \ge n$. So the above estimate implies 
$$ \|\varphi^{2n}H\|_{L^{p_0} (B_{\frac{1}{2}}(0))} \leq C_4 \int_{B_{1}(0)}H^{n+1}\varphi^{n}dx, $$
where $C_4= C(n) C_3^{C(n)}\cdot C(\|f^{-1}\|_{L^\infty(B_1(0))},\|\varphi\|_{L^\infty(B_1(0))})$. Let $k_0\rightarrow +\infty$, that is $p_0 \to +\infty$, we have the conclusion when $n> 2$. The case $n=2$ is simply the Sobolev embedding theorem that $W^{1,2}\hookrightarrow C^0$. 
\end{proof}

\subsection{Proof to Theorem \ref{Thm_kappa}}
\begin{proof}
We refer to Mooney \cite{MR4246798} that the Pogorelov-type estimate for $\sigma_2$ equation results in a substantial interior Hessian bound for convex solutions. Denote $ W:= \sqrt{1+|Du|^2}.$   
Given a function $v$ in $\mathbb R^n$, we introduce the following symmetric matrix, see Caffarelli-Nirenberg-Spruck \cite{Caffarelli1986starshape}:
\begin{eqnarray*}
H_v =(I_n-\frac{Du\otimes Du}{W(1+W)}) {D^2 v} (I_n-\frac{Du\otimes Du}{W(1+W)}),
\end{eqnarray*}
where $D^2 v$ is the Hessian matrix and $u$ is the graph function for $M$. It can be checked that $H_v$ has the same eigenvalues as $g^{-1}D^2v$. Denote 
$$ P = I_n-\frac{Du\otimes Du}{W(1+W)}.$$
Then $P$ is positive definite and $ 0 < c_0 \leq \|P\| \leq 1$ for some $c_0=c_0(\|Du\|_{L^\infty (B_2(0))})>0$.
Suppose $D^2 u$ and $H_u$ has eigenvalues $\lambda_1 \ge \cdots \ge \lambda_n \ge 0 $ and $\lambda'_1 \ge \cdots \ge \lambda'_n \ge 0 $ respectively. Recall the Courant minimax principle for symmetric matrix, we have 
\begin{align*}
    \lambda'_k & = \max_{dim(S)=k} \min_{0\neq x \in S} \frac{\langle H_u x,x \rangle}{\|x\|^2} = \max_{dim(S)=k} \min_{0\neq x \in S} \frac{\langle(D^2 u)Px,Px\rangle}{ \|x\|^2} \\
    & \le \max_{dim(S)=k} \min_{0\neq x \in S} \frac{\langle(D^2 u) Px,Px\rangle}{\|Px\|^2} = \max_{dim(S)=k} \min_{0\neq y \in S} \frac{\langle(D^2 u)\, y,y\rangle}{\|y\|^2} = \lambda_k,
\end{align*}
and therefore 
$$
\sigma_2(D^2 u) \ge \sigma_2(H_u) =  W^2 \sigma_2(\kappa) = W^2 f \ge W^2 f_0 > 0.
$$ 
Following Mooney's proof of Proposition 4.1 in \cite{MR4246798}, we have for any $r>0$, there exists $\delta = \delta(n,\|u\|_{C^{0,1}(B_1)},r,f_0)>0$ such that if $L$ is a supporting linear function to $u$ at $0$, then 
$$ \{ u < L + \delta \} \subset\subset T_r$$
for some $n-2$-dimensional subspace $T$ of $\mathbb{R}^n$, where $T_r$ denotes the $r$-neighbourhood of $T$.

Take $L(x)= u(0)+x \cdot Du(0)$, then there exists a constant $ \delta >0$ such that after rotation, $ u > L + \delta$ on $\{|(x_1,x_2)|=r\}\cap B_1(0)$. Consider the following function
$$ w(x)=\delta[M(x_1^2+x_2^2)- (x_3^2+\cdots+x_n^2) + \frac{1}{4}] + L(x). $$
Denote $\mu_1 \geq \cdots \geq \mu_n$ and $\mu'_1 \geq \cdots \geq \mu'_n$ the eigenvalues of $D^2 w$ and $H_w$ respectively. We can see that
$$ \mu_1 = \mu_2 = 2\delta M >0, \ \ \mu_3 = \cdots \mu_n = -2\delta <0.$$
We have
\begin{align*}
    \mu'_1 & = \max_{x\in\mathbb{R}^n} \frac{\langle H_w x,x \rangle}{\|x\|^2}\\
    & \ge \frac{\langle P(D^2 w)PP^{-1}E_1,P^{-1}E_1\rangle}{\,\|P^{-1}E_1\|^2} = \frac{\langle (D^2 w) E_1, E_1 \rangle}{\,\|P^{-1}E_1\|^2} \\
    & \ge C \mu_1 = 2C\delta M >0.
\end{align*}
where $C=C(\|u\|_{C^{0,1}(B_2(0))})>0$.  Similarly,
\begin{align*}
    \mu'_2 & = \max_{dim(S)=2} \min_{0\neq x \in S} \frac{\langle H_w x,x \rangle}{\|x\|^2}\\
    & \ge \min_{0\neq x \in span\{P^{-1}E_1,P^{-1}E_2\}} \frac{\langle H_w x,x\rangle }{\|x\|^2} \ge C \mu_1 = 2C\delta M >0.
\end{align*}
As for the smallest eigenvalue, we have
\begin{align*}
    \mu'_n & = \min_{x\in\mathbb{R}^n} \frac{\langle H_w x,x\rangle}{\|x\|^2}\\
    & \le \frac{\langle P(D^2 w)PP^{-1}E_n,P^{-1}E_n\rangle }{\,\|P^{-1}E_1\|^2} = \frac{\langle (D^2 w) E_n, E_n\rangle}{\,\|P^{-1}E_1\|^2} <0.
\end{align*}
Take $y\in \mathbb{R}^n$ satisfies $H_w y = \mu'_n y$ and $\|y\|=1$, then we have
\begin{align*}
    0 > \mu'_n & =\langle H_w y,y \rangle\ = \langle D^2 w Py,Py\rangle\\
    &  \ge\ \frac{\langle D^2 w Py,Py\rangle}{\|Py\|^2} \ge \min_{x\in\mathbb{R}^n} \frac{\langle(D^2w) x,x\rangle}{\|x\|^2} = \mu_n = -2\delta .
\end{align*}
So we have 
\begin{align*}
    \sigma_2(H_w) & = \sigma_2(\mu'_1,\cdots,\mu'_n) \\ 
    & \ge \mu'_1 \mu'_2 - C(n)\mu'_1 |\mu'_n| = \mu_1'[2CM\delta- C(n) \delta ] \ge 0
\end{align*}
if we choose $M = M(\|u\|_{C^{0,1}(B_2(0))}) > C(n)C^{-1}$. Subsequently, we take $r = \sqrt{\frac{1}{4(M+1)}}$. We can check that this $w$ meets our requirements:
\begin{enumerate}
    \item $ u(0) < w(0) $, since $\frac{1}{4}>0$,
    \item $\sigma_2(g^{-1}D^2 w) \ge 0 $,
    \item $ w-L < \delta$ on $\{|(x_1,x_2)|=r\}\cap B_1(0)$, that is $Mr^2 + \frac{1}{4} < 1$,
    \item $w-L < 0 $ on $\{|(x_1,x_2)|\leq r\}\cap \partial B_1(0)$, that is $(M+1)r^2 + \frac{1}{4} < 1$. 
\end{enumerate}
It follows that $u>w$ on the boundary of $\{|(x_1,x_2)|\le r\}\cap B_1(0)$. Let $\Omega$ denote the connected component of the set $\{u<w\}$ which contains $x=0$. Take 
$$
    \varphi(x):=
    \begin{cases}
        [w(x)-u(x)]^{4} \ \ \ & \text{if}\ x \in \Omega \\
        \ 0  & \text{if}\ x \in B_2(0)\backslash\Omega.
    \end{cases}
$$

Applying Proposition \ref{lemma:curvature_W2p_estimate} and Proposition \ref{lemma:curvarure_C11ptype_estimate}, we get
$$
|H(0)|\le C(n, \delta,\|u\|_{C^{0,1}(B_2(0))},\|f\|_{C^{0,1}(B_2(0))},\|f^{-1}\|_{L^{\infty}(B_{2}(0))}),
$$
the conclusion follows.
\end{proof}

\section{Estimates for Hessian equations}
In this section we deal with the Hessian equation case. Similar to the curvature equation, we can introduce divergence structure for $\sigma_2$ operator to integrate the Jacobi inequality by parts, we also refer to Shankar-Yuan \cite{shankar2020hessian} for this technique.
Let
$$
F(D^2u)=\sigma_2(D^2u)=\frac12[(\Delta u)^2-|D^2u|^2]=f,
$$
denote $[F^{ij}]_{n\times n}=\Delta u I_n-D^2u$, the following notations will be used throughout this section,
$$
\Delta_F v=\sum^n_{i,j=1}F^{ij}\partial_{ij}v=\sum^n_{i,j=1}\partial_j(F^{ij}\partial_{i}v), \ |\nabla_F v|^2=\sum^n_{i,j=1}F^{ij}D_i vD_j v.
$$
\begin{lemma}\label{lemma:Heesian_Trace_Jacobi}
Let $u$ be a smooth convex solution to equation (\ref{eq:sigma_lambda}). Suppose that $f$ is smooth, $f\ge 1$. Then for each constant $\epsilon\in(0,1)$ we have
\begin{equation}\label{eq:Hessian_Trace_Jacobin1}
   \begin{aligned}
\Delta_F\Delta u-(2-\epsilon)(\Delta u)^{-1}|\nabla_F\Delta u|^2\ge{\Delta f}-C(\epsilon)|Df|^2.
   \end{aligned}
\end{equation}
\end{lemma}
\begin{proof}
Differentiate equation (\ref{eq:sigma_lambda}) with respect to $x_k$ twice,
$$
\sum^n_{i,j=1}\partial_k F^{ij} \partial_{ij}u_k+\sum^n_{i,j=1}F^{ij}\partial_{ij}u_{kk}=f_{kk}.
$$
Summing up the equality we get
\begin{align*}
    \sum^n_{i,j=1}F^{ij}\partial_{ij}\Delta u &= -\sum^n_{i,j,k=1}\partial_k F^{ij}\partial_{ij}u_k+\Delta f
    =-\sum^n_{i,j,k=1}(\Delta u_k\delta_{ij}-u_{ijk})u_{ijk}+\Delta f\\
    &=\sum^n_{i,j,k=1}u^2_{ijk}-\sum^n_{k=1}(\Delta u_k)^2+\Delta f\ge 3\sum^n_{i\neq k}u^2_{iik}+\sum^n_{k=1}u^2_{kkk}-\sum^n_{k=1}(\Delta u_k)^2+\Delta f.
\end{align*}
Let $\delta\ge0$ be a nonegative constant, choose a coordinate system such that $D^2u$ is diagonal at a point, where we have
\begin{equation}\label{eq:DeltaF_convex}
\begin{aligned}
\Delta_F\Delta u -\delta\frac{|\nabla_F \Delta u|^2}{\Delta u}
\ge
3\sum^n_{i\neq k}u^2_{iik}+\sum^n_{k=1}u^2_{kkk}-\sum^n_{k=1}(1+\delta\frac{\Delta u-\lambda_k}{\Delta u})(\Delta u_k)^2+\Delta f.
\end{aligned}
\end{equation}
Differentiate equation (\ref{eq:sigma_lambda}) once we have $\sum^n_{i,j=1}F^{ij}u_{ijk}=f_k$ for each $1\le k\le n$.
Let
\begin{equation}\label{eq:Qk_convex}
Q_k=3\sum^n_{i\neq k}u^2_{iik}+u^2_{kkk}-(1+\delta\frac{\Delta u-\lambda_k}{\Delta u})[\sum^n_{i=1}(1-\frac{\Delta u-\lambda_i}{\Delta u})u_{iik}+\frac{f_k}{\Delta u}]^2. 
\end{equation}
By the diagonalize lemma mentioned in Lemma \ref{lemma:curvature_Trace_Jacobi}, we want to show that
$$
\hat Q_k=1-\frac1{3}\sum_{i\neq k}(1+\delta\frac{\Delta u-\lambda_k}{\Delta u})(1-\frac{\Delta u-\lambda_i}{\Delta u})^2
-(1+\delta\frac{\Delta u-\lambda_k}{\Delta u})(1-\frac{\Delta u-\lambda_k}{\Delta u})^2\ge 0.
$$
We have
\begin{align*}
    \hat Q_k&\ge1-\frac1{(\Delta u)^2}[\frac1{3}(1+\delta)\sum^n_{i\neq k}\lambda_i^2+\lambda_k^2+\delta\frac{\lambda^2_k(\Delta u-\lambda_k)}{\Delta u}]\\
    &\ge1-\frac1{(\Delta u)^2}[\frac1{3}(1+\delta)\sum^n_{i\neq k}\lambda_i^2+\lambda_k^2+\delta f].
\end{align*}
Clearly $\hat Q_k\ge0$ if we choose $\delta=2-\frac\epsilon2$, since $(\Delta u)^2= \sum^n_{i=1}\lambda_i^2+2f$. By (\ref{eq:Qk_convex}) we  get
\begin{align*}
Q_k\ge -2(1+\delta\frac{\Delta u-\lambda_k}{\Delta u})[\sum^n_{i=1}(1-\frac{\Delta u-\lambda_i}{\Delta u})u_{iik}]\frac{f_k}{\Delta u}-(1+\delta\frac{\Delta u-\lambda_k}{\Delta u})(\frac{f_k}{\Delta u})^2.
\end{align*}
Recall that $\Delta u\ge1$, we choose a small constant $\hat\epsilon=\epsilon/{2}$,
\begin{equation}\label{eq:Q1_convex}
Q_k\ge -\frac{\hat\epsilon}{(\Delta u)^2}(\Delta u_k)^2-\frac{C(1)}{\hat\epsilon}|Df|^2.
\end{equation}
By (\ref{eq:DeltaF_convex}) (\ref{eq:Qk_convex}) (\ref{eq:Q1_convex}) we get
$$
\Delta_F\Delta u -(2-\epsilon)\frac{|\nabla_F \Delta u|^2}{\Delta u}\ge \Delta f- C(\epsilon)|Df|^2.
$$
\end{proof}

\begin{lemma}[Boundary Jacobi inequality]\label{lemma:Hessian_BTrace_Jacobi}
Let $u$ be a smooth convex solution to equation (\ref{eq:sigma_lambda}) in $B_2(0)$.
 Let $\varphi(t)=(t^+)^4 $, $t\in \mathbb{R}^1$ be a $C^2$ function. Suppose that $f$ is smooth, $f\ge 1$. Then for each smooth function $w(x)$, 
\begin{align*}
\Delta_F [\varphi(w-u)\Delta u ]
&\ge [\varphi'(w-u)\Delta_F(w-u)]\Delta u + [{\Delta f}-C(1)|Df|^2]\varphi(w-u).
\end{align*}

\end{lemma}
\begin{proof}
We begin by direct calculation at the points that $w-u>0$,
\begin{align*}
\Delta_F[\varphi(w-u)]& =\varphi''(w-u)|\nabla_F (w-u)|^2+\varphi'(w-u)\Delta_F(w-u),\\
|\nabla_F[\varphi(w-u)]|^2 & = [\varphi'(w-u)]^2|\nabla_F (w-u)|^2.
\end{align*}
Then we get
\begin{align*}
&\Delta_F[\varphi(w-u)]-\frac23|\nabla_F[\varphi(w-u)]|^2[\varphi(w-u)]^{-1} \\
&=\{\varphi''(w-u)-\frac23[\varphi(w-u)]^{-1}[\varphi'(w-u)]^2\}|\nabla_F (w-u)|^2+\varphi'(w-u)\Delta_F(w-u)\\
&\ge \varphi'(w-u)\Delta_F(w-u).
\end{align*}
We also know from Lemma \ref{lemma:Heesian_Trace_Jacobi} that
$$
\Delta_F\Delta u -\frac{3}{2}(\Delta u)^{-1}{|\nabla_F \Delta u|^2}\ge \Delta f- C(1)|Df|^2.
$$
Now we consider the boundary Jacobi inequality for the function $\varphi(w-u)\Delta u$ which vanishes at the level set $w=u$, and choose $\epsilon=\frac12$ in inequality (\ref{eq:Hessian_Trace_Jacobin1}),
\begin{align*}
\Delta_F[\varphi\Delta u]
&= \Delta u\Delta_F\varphi+\varphi\Delta_F\Delta u+2\sum^n_{i,j=1}F^{ij}\partial_i\varphi\partial_j\Delta u
\\
&\ge\{\Delta_F\varphi-\frac23\varphi^{-1}|\nabla_F\varphi|^2\}\Delta u
+\{\Delta_F\Delta u -\frac32(\Delta u)^{-1}{|\nabla_F \Delta u|^2}\}\varphi\\
&\ge [\varphi'\Delta_F(w-u)]\Delta u+[\Delta f- C(1)|Df|^2]\varphi.
\end{align*}

\end{proof}

\subsection{Pogorelov-type \texorpdfstring{$W^{2,p}$}{} estimates}

\begin{definition}
We say that a $C^2$ function $w$ in $\mathbb{R}^n$ is $2$-convex, if 
$$ \sigma_k(D^2w)>0, \ \ k=1,2. $$
\end{definition}

\begin{prop}\label{lemma:Hessian_W2p_estimate}
 Suppose $u$ is a smooth convex solution to equation (\ref{eq:sigma_lambda}) in $  B_2(0)$. Let $w$ be a $2$-convex function in $ \mathbb{R}^n $, and $\Omega \subset B_1(0)$ satisfies $w\le u$ on $\partial \Omega$. Define $\varphi(t)$ as in Lemma \ref{lemma:Hessian_BTrace_Jacobi}. Suppose that $f\ge1$, $f$ is smooth in $B_2(0)$. Then for $p=1,2,\cdots$,
$$
    \int_{\Omega}[\varphi(w-u)]^{p-1}(\Delta u)^p dx \le p!C(n,\|w\|_{C^{0,1}(B_2(0))},\|u\|_{C^{0,1}(B_2(0))},\|f\|_{C^{0,1}(B_2(0))}).
$$
\end{prop}
\begin{proof}
Let $p\ge 1$ be an integer, we choose $[\varphi(w-u)\Delta u]^p$ as a test function. By Lemma \ref{lemma:Hessian_BTrace_Jacobi} we have
\begin{equation}\label{eq_temp_1}
\begin{aligned}
 &p\int_{\Omega}|\nabla_F[\varphi \Delta u]|^2[\varphi \Delta u]^{p-1} dx\\
 &\le -\int_{\Omega}[\varphi' \Delta_F (w-u)](\Delta u)^{p+1}\varphi ^p
 -\int_{\Omega}[\Delta f- C(1)|Df|^2](\Delta u)^p\varphi ^{p+1} dx.
\end{aligned}
\end{equation}
Choose a small constant $0<\delta<1$, we integrate by parts to estimate that
\begin{align*}
    &\int_{\Omega}(\Delta u)^{p+1}\varphi ^pdx=\int_{\Omega}[\varphi \Delta u]^{p}\Delta u dx
    =-\int_{\Omega}D([\varphi \Delta u]^{p})Du dx\\
    &\le \delta p\int_{\Omega} |D[\varphi \Delta u]|^2[ \Delta u]^{p-2}\varphi^{p-1} dx
    +\frac p{\delta} \int_{\Omega}|Du|^2(\Delta u)^{p}\varphi ^{p-1}dx,
\end{align*}
\begin{align*}
&\int_{\Omega}\Delta f(\Delta u)^{p}\varphi ^{p+1}dx=
-\int_{\Omega}\varphi DfD([\varphi \Delta u]^{p}) dx-\int_{\Omega}[\varphi \Delta u]^{p}DfD\varphi  dx\\
&\le \delta p\int_{\Omega} |D[\varphi \Delta u]|^2[ \Delta u]^{p-2}\varphi^{p-1} dx
    +\frac p\delta C_1 \int_{\Omega}(\Delta u)^{p}\varphi ^{p-1}dx,
\end{align*}
where $C_1=C(n,\|\varphi\|_{C^{0,1}(\Omega)}, \|Df\|_{L^\infty(\Omega)})$. Recall $(\Delta u-\lambda_i)\Delta u\ge f\ge 1$ for each $1\le i\le n$, thus
$$
|\nabla_F[\varphi \Delta u]|^2[\varphi \Delta u]^{p-1}\ge |D[\varphi \Delta u]|^2[ \Delta u]^{p-2}\varphi^{p-1} .
$$
Notice also that $\Delta_F u=2f\le 2\|f\|_{L^\infty(\Omega)}$, and since $w$ is 2-convex.
$$
\Delta_F w=\mathrm{Trace}([F_{ij}]_{n\times n}D^2w)= \Delta w\Delta u-\mathrm{Trace}(D^2wD^2u)\ge0,
$$
We get $ \varphi'\Delta_F(w-u) \ge -2\|\varphi\|_{C^{0,1}(\Omega)}\|f\|_{L^\infty(\Omega)}=-C_2$. Let $\delta=\frac1{4C_2+1}$, using (\ref{eq_temp_1}) we get the recursion formula for $p\ge 1$,
$$
\int_{\Omega}(\Delta u)^{p+1}\varphi ^pdx\le pC_*\int_{\Omega}(\Delta u)^{p}\varphi ^{p-1}dx,
$$
where $C_*=C_*(C,C_1,C_2,\|f^{-1}\|_{L^\infty(\Omega)})$.
When $p=1$, we choose a cutoff function $\phi \in C_0^{\infty}(B_2(0)),\phi =1 $ in $B_1(0)$, then 
\begin{eqnarray*}
    \int_\Omega \Delta u dx \leq \int_{B_2} \phi^2 div(Du) dx \leq 2 \int_{B_2} \frac{D\phi \cdot Du}{W} dx \leq C(\|Du\|_{L^{\infty}(B_2(0)}).
\end{eqnarray*}
\end{proof}

\subsection{Pogorelov-type \texorpdfstring{$C^{1,1}$}{} estimates}

\begin{prop}\label{lemma:Hessian_C11ptype_estimate}
Let $w$ be a smooth function in $\mathbb R^n$. Let $u$ be a smooth convex solution to equation (\ref{eq:sigma_lambda}) in $B_2(0) $. Define $\varphi(t)$ as in Lemma \ref{lemma:Hessian_BTrace_Jacobi}. Suppose that $f\ge1$ in $ \overline{B_2(0)} $, $f$ is smooth. Then
$$
    \|[\varphi(w-u)]^{2n}\Delta u\|_{L^{\infty} (B_{\frac12}(0))}\le C \int_{B_{1}(0)}\varphi^{n} (\Delta u)^{n+1}dx,
$$
where $C=C(n,\|w-u\|_{C^{0,1}(B_1(0))},\|f\|_{C^{0,1}(B_1(0))})$.
\end{prop}
\begin{proof}
Let $0<\rho<\frac12$, choose a cut-off function $\phi\in C^\infty_0(B_{\frac{1}{2}+\rho}(0))$ with $\phi=1$ in $B_{\frac{1}{2}}(0)$, $\phi\ge 0$, $|D\phi|\le 8\rho^{-1} $. Let $p\ge 1,q\ge 2$ be a pair of positive constants satisfying $1\le\frac qp\le 2n$. According to Lemma \ref{lemma:Heesian_Trace_Jacobi}, $\Delta u$ is subharmonic, we have
$$
\int_{B_1(0)}(\Delta u)^p[\varphi(w-u)]^{q}\phi^2\Delta_F\Delta u dx \ge\int_{B_1(0)}[{\Delta f}-C(1)|Df|^2](\Delta u)^p[\varphi(w-u)]^{q}\phi^2 dx.
$$
Integrating by parts the above inequality and recalling $\Delta u-\lambda_i\le \Delta u$, we get
\begin{align*}
\int_{B_{{\frac12}+\rho}(0)}|D\Delta u|^2(\Delta u)^{p-2}\varphi^q\phi^2dx
& \le C_1\int_{B_{{\frac12}+\rho}(0)}(|D\phi^2|+|Df|^2+|D\varphi|^2)(\Delta u)^{p+2}\varphi^{q-2}dx \\
& \le \frac{C_2}{\rho^2} \int_{B_{{\frac12}+\rho}(0)} (\Delta u)^{p+2}\varphi^{q-2} dx,
\end{align*}
where $C_1=C(\|\varphi\|_{C(B_1)}) $, and $C_2= C_1\cdot( 1+\|f\|^2_{C^{0,1}(B_1)}+\|\varphi\|^2_{C^{0,1}(B_1)})$.
Then we have
\begin{align*}
\int_{B_{\frac12}(0)}|D[(\Delta u)^{\frac p2}\varphi^{\frac q2}]|^2dx
\le \frac{C_2p^2q^2}{\rho^2}\int_{B_{{\frac12}+\rho}(0)}(\Delta u)^{p+2}\varphi^{q-2}dx,
\end{align*}
and
\begin{align*}
\int_{B_{\frac{1}{2}}(0)}|(\Delta u)^{\frac p2}\varphi^{\frac q2}|^2 dx
\le C_1\int_{B_{{\frac{1}{2}}+\rho}(0)}(\Delta u)^{p+2}\varphi^{q-2}dx \le \frac{C_2p^2q^2}{\rho^2}\int_{B_{{\frac{1}{2}}+\rho}(0)}(\Delta u)^{p+2}\varphi^{q-2}dx .
\end{align*}
For $n>2$, we get an iteration formula by Sobolev embedding theorem
\begin{align*}
\int_{B_{\frac12}(0)}(\Delta u)^{\frac {np}{n-2}}\varphi^{\frac {nq}{n-2}} dx
&\le [\frac{C_2p^2q^2}{\rho^2}\int_{B_{{\frac12}+\rho}(0)}(\Delta u)^{p+2}\varphi^{q-2}dx]^{\frac {n}{n-2}} \\
& \le [\frac{C_3 p^4}{\rho^2}\int_{B_{{\frac{1}{2}}+\rho}(0)}(\Delta u)^{p+2}\varphi^{q-2}dx]^{\frac {n}{n-2}},
\end{align*}
where $ C_3 =  4n^2 C_2 $, since we have required $\frac{q}{p} \le 2n $. Denote $\gamma=\frac{n}{n-2}$, we fix an integer $k_0\ge\frac{\ln n}{\ln\gamma}$, and take $p_0 = \gamma^{k_0}$, $q_0 = 2np_0$ to initiate the iteration. For $k = 1,\cdots,k_0$, let
\begin{eqnarray*}
    p_{k} &=& \gamma^{-1} p_{k-1}+2=\gamma^{-k}p_0+2\sum^{k-1}_{i=0}\gamma^{-i}, \\
    q_{k} &=& \gamma^{-1} q_{k-1}-2=\gamma^{-k}q_0-2\sum^{k-1}_{i=0}\gamma^{-i}, \\
    r_k &=& \frac{1}{2}+\sum^{k}_{i=1}2^{-(k_0-i+2)}.
\end{eqnarray*}

We can check that $p_k > n$ and $q_k \ge 2 $ satisfy $\frac{q_k}{p_k} \le 2n$ all along.
Then we can rewrite the formula as
$$
\int_{B_{r_{k-1}}}(\Delta u)^{p_{k-1}}\varphi^{q_{k-1}}dx 
\le [C_3\, 2^{2(k_0-k+2)} (p_k-2)^4 \int_{B_{r_{k}}}(\Delta u)^{p_k}\varphi^{q_{k}}dx]^{\gamma}.
$$
Therefore, an iteration process implies
\begin{align*}
\int_{B_{\frac{1}{2}}(0)}(\Delta u)^{p_0}\varphi^{q_{0}}dx
&\le C_3^{\sum^{k_0}_{k=1}\gamma^k}2^{2\sum^{k_0}_{k=1}(k_0-k+2)\gamma^k}\prod ^{k_0}_{k=1}(p_k-2)^{4\gamma^k}[\int_{B_{r_{k_0}}(0)}(\Delta u)^{p_{k_0}}\varphi^{q_{k_0}}dx]^{\gamma^{k_0}}\\
&\le C_3^{\sum^{k_0}_{k=1}\gamma^k} 4^{\sum^{k_0}_{k=1}(k_0-k+2)\gamma^k} \prod^{k_0}_{k=1}(n\gamma^{-k}p_0)^{4\gamma^k} [\int_{B_{r_{k_0}}(0)}(\Delta u)^{p_{k_0}}\varphi^{q_{k_0}}dx]^{\gamma^{k_0}} \\
& = (n^4 C_3)^{\sum^{k_0}_{k=1}\gamma^k} 4^{\sum^{k_0}_{k=1}(k_0-k+2)\gamma^k} \gamma^{4\sum^{k_0}_{k=1}(k_0-k)\gamma^k} [\int_{B_{r_{k_0}}(0)}(\Delta u)^{p_{k_0}}\varphi^{q_{k_0}}dx]^{p_0}.
\end{align*}

Notice that $\frac{1}{\gamma^{k_0}} \sum^{k_0}_{k=1}\gamma^k $ and $ \frac{1}{\gamma^{k_0}} \sum^{k_0}_{k=1}(k_0-k)\gamma^k $ can both be bounded by some constant $C(n)$. So we get
\begin{align*}
\|\varphi^{2n}\Delta u\|_{L^{p_0} (B_{\frac{1}{2}}(0))} = [\int_{B_{\frac{1}{2}}(0)}(\Delta u)^{p_0}\varphi^{q_{0}}dx]^\frac 1{p_0}
&\le C(n)\cdot C_3^{C(n)} \int_{B_{1}(0)}(\Delta u)^{p_{k_0}}\varphi^{q_{k_0}}dx.
\end{align*}
From the choice of $p_0, q_0 $ and $k_0$, we can verify that $p_{k_0} \le n+1 $ and $q_{k_0} \ge n$. So the above estimate implies 
$$ \|\varphi^{2n}\Delta u\|_{L^{p_0} (B_{\frac{1}{2}}(0))} \leq C_4 \int_{B_{1}(0)}(\Delta u)^{n+1}\varphi^{n}dx, $$
where $C_4= C(n) C_3^{C(n)}\cdot C(\|\varphi\|_{L^\infty(B_1(0))})$. Let $k_0\rightarrow +\infty$, that is $p_0 \to +\infty$, we have the conclusion when $n> 2$. The case $n=2$ is simply the Sobolev embedding theorem that $W^{1,2}\hookrightarrow C^0$. 
\end{proof}

\subsection{Proof to Theorem \ref{Thm_sigma}}
\begin{proof}
We refer to Mooney \cite{MR4246798}. We first suppose that $f\ge1$, and denote $\hat u=u-u(0)-x Du(0)$. Then $\sigma_2(D^2\hat u)=f\ge1$, $\hat u\ge 0$. By \cite{MR4246798} Proposition 4.1, there exists a constant $ \delta=\delta (n,\|\hat u\|_{L^\infty(B_1(0))})>0$, such that after a rotation, $\hat u> \delta$ on $\{|(x_1,x_2)|=\frac1{2n}\}\cap B_1(0)$. Consider the 2-convex function in $\mathbb{R}^n$
 $$
 w(x)=\delta[2(n-2)|(x_1,x_2)|^2-|(x_3,\cdots,x_n)|^2+\frac18].
 $$
Then $u>w$ on $\partial B_{\frac34}(0)$. The estimate follows by applying Proposition \ref{lemma:Hessian_W2p_estimate} and Proposition \ref{lemma:Hessian_C11ptype_estimate} on the connected component of the set $\{ x\, |\, u(x)<w(x) \}$ which contains $x=0$,
$$
|D^2 \hat u(0)|\le C(n, \delta,\|u\|_{C^{0,1}(B_2(0))},\|f\|_{C^{0,1}(B_2(0))}).
$$
Finally for general $f>0$, we scale by $\overline u=u\cdot \|f^{-1}\|_{L^{\infty}(B_{1}(0))}$ and the conclusion follows.
\end{proof}

\bibliography{References}

\end{document}